\documentclass[11pt,a4paper]{article}
\usepackage{amsmath}
\usepackage{amssymb}
\usepackage{amsthm}
\usepackage{latexsym}
\usepackage{amsbsy}
\usepackage[all]{xypic}
\usepackage{amscd}
\usepackage{mathrsfs}
\usepackage{tipa}
\usepackage{txfonts}
\usepackage{amsfonts}

\newtheorem{thm}{Theorem}[section]
\newtheorem{cor}[thm]{Corollary}
\newtheorem{lem}[thm]{Lemma}

\newtheorem{prop}[thm]{Proposition}
\newtheorem{defn}[thm]{Definition}
\newtheorem{rem}[thm]{Remark}

\begin{document}

\begin{center}
{\Large \bf Maximal rigid subcategories in $2-$Calabi-Yau
triangulated categories \footnote{Supported by the NSF of China
(Grants 10771112) }

 \small{Dedicated to Claus M. Ringel on the occasion of his $65^{th}$ birthday }}
\bigskip

{\large Yu Zhou and
 Bin Zhu}
\bigskip

{\small
\begin{tabular}{cc}
Department of Mathematical Sciences & Department of Mathematical
Sciences
\\
Tsinghua University & Tsinghua University
\\
  100084 Beijing, P. R. China &   100084 Beijing, P. R. China
\\
{\footnotesize E-mail: yu-zhou06@mails.tsinghua.edu.cn} &
{\footnotesize E-mail: bzhu@math.tsinghua.edu.cn}
\end{tabular}
}
\bigskip


\end{center}

\def\s{\stackrel}
\def\Longrightarrow{{\longrightarrow}}
\def\A{\mathcal{A}}
\def\C{\mathcal{C}}
\def\D{\mathcal{D}}
\def\T{\mathcal{T}}
\def\R{\mathcal{R}}
\def\S{\mathcal{S}}
\def\H{\mathcal{H}}
\def\U{\mathscr{U}}
\def\V{\mathscr{V}}
\def\W{\mathscr{W}}
\def\X{\mathscr{X}}
\def\Y{\mathscr{Y}}
\def\Z{\mathcal {Z}}

\def\add{\mbox{add}}
\def\Aut{\mbox{Aut}}
\def\Hom{\mbox{Hom}}
\def\Ext{\mbox{Ext}}
\def\ind{\mbox{ind}}
\def\deg{\mbox{deg}}
\def \text{\mbox}

\def\add{\mbox{add}}
\def\Aut{\mbox{Aut}}
\def\End{\mbox{End}}
\def\Ext{\mbox{Ext}}
\def\deg{\mbox{deg}}
\def\Hom{\mbox{Hom}}
\def\ind{\mbox{ind}}
\def\dim{\mbox{dim}}
\def \text{\mbox}
\hyphenation{ap-pro-xi-ma-tion}

\begin{abstract}

We study the functorially finite maximal rigid subcategories in
$2-$CY triangulated categories and their endomorphism algebras.
Cluster tilting subcategories are obviously functorially finite and
maximal rigid; we prove that the converse is true if the $2-$CY
triangulated categories admit a cluster tilting subcategory.
 As a generalization of a result of [KR], we prove that any functorially finite maximal rigid subcategory is Gorenstein
 with Gorenstein dimension at most $1$. Similar as cluster tilting subcategory, one can mutate maximal rigid
 subcategories at any indecomposable object.
  If two maximal rigid objects are reachable via mutations, then their endomorphism algebras have the same representation
  type.
\end{abstract}

\medskip


\section{Introduction} In the categorification theory of cluster
algebras [FZ], cluster categories [BMRRT, K1, Am], (stable) module
categories over preprojective algebras [GLS1, GLS2, BIRS], and more
general $2-$Calabi-Yau triangulated categories with cluster tilting
objects [FuKe, Pa1] play a central role. We refer the reader to the
nice surveys [GLS, K2, BM, Rin] and the references there for the
recent developments.
\medskip

Cluster tilting objects (subcategories) in $2-$CY categories have
many nice properties. For examples, the endomorphism algebras are
Gorenstein algebras of dimension at most $1$ [KR]; cluster tilting
objects have the same number of non-isomorphic indecomposable direct
summands [DK, Pa2]. Importantly, in the categorification of cluster
algebras, cluster tilting objects categorify clusters of the
corresponding cluster algebras, and the combinatorics structure of
cluster tilting objects is the same as the combinatorics structure
of the corresponding cluster algebras [CC, CK].
\medskip

  Cluster tilting objects (subcategories) are maximal rigid objects (subcategories), the converse is not true in general.
  The first examples of $2-$Calabi-Yau
 categories in which maximal rigid objects are not cluster tilting
 were given in [BIKR] (see also the example in Section 5 of [KZ] for the example of triangulated category in which maximal objects are not cluster tilting).
  Cluster tubes introduced in [BKL] serves as another type of such examples. It was proved recently by Buan-Marsh-Vatne in [BMV] that cluster tubes contain maximal rigid objects,
 none of them are cluster tilting. Buan-Marsh-Vatne also proved that
   the set of maximal
 rigid objects in $2-$CY triangulated categories forms cluster structure satisfying the definition in [BIRS] by allowing
 loops, and the combinatorial structure of maximal rigid objects in a cluster
 tube models the combinatorics of a type $B$ cluster algebra. In
 [V, Y], the authors studied the endomorphism algebras of maximal rigid
 objects in cluster tubes, in particular, they proved that the
 endomorphism algebras are Gorenstein of Gorenstein dimension at
 most $1$.
\medskip

The aim of this paper is to give a systematic study of functorially
finite maximal rigid subcategories in $2-$CY triangulated categories
and endomorphism algebras of maximal rigid objects. For any
functorially finite maximal rigid subcategory $\R$ in a $2-$CY
triangulated category $\C$, one considers the extension subcategory
$\R *\R[1]$ (compare [Pla]). It is not equal to $\C$. Note that
under the condition that $\R$ is cluster tilting, we have $\C=\R*\R
[1]$ [KR].  We observe that any rigid object belongs to $\R*\R[1]$.
Using this fact, we prove that if a $2-$CY triangulated category
contains a cluster tilting subcategory, then any functorially finite
maximal rigid subcategory is cluster tilting. This generalizes
Theorem II.1.8 in [BIRS] from algebraic $2-$CY triangulated
categories to arbitrary $2-$CY triangulated categories. Then we
consider $2-$CY triangulated categories with maximal rigid
subcategories. It is proved that some results in [DK, Pa] also hold
in this setting. Namely, we prove that the representatives of
isomorphic classes of indecomposable objects of a functorially
finite maximal rigid subcategory form a basis in the split
Grothendieck group of another functorially finite maximal rigid
subcategory. In particular, maximal rigid objects have the same
number of non-isomorphism indecomposable direct summands. Using a
recent result of Nakaoka [Na], we prove that functorially finite
maximal rigid subcategories in $2-$CY triangulated categories are
Gorenstein of dimension at most $1$. This is a generalization of the
same results in cluster tubes[V, Y]. And it also generalizes
 the same results on cluster tilting subcategories of [KR] to functorially finite maximal rigid
 subcategories.
  Finally, we study the endomorphism algebras of maximal rigid objects
in $2-$CY triangulated categories. We call such algebras
 $2-$CY maximal tilted algebras. If two maximal rigid objects are reachable
 via simple mutations, then the corresponding $2-$CY maximal
 tilted algebras have the same representation type.
\medskip

The paper is organized as follows.
\medskip

 In Section 2, we prove that the
functorially finite maximal rigid subcategories are cluster tilting
in $2-$CY triangulated categories with a cluster tilting
subcategory. In Section 3, the notion of index of a rigid object
with respect to a cluster tilting subcategory is generalized by
replacing cluster tilting subcategory with functorially finite
maximal rigid subcategory
 with respect to maximal rigid subcategory (compare [Pla]). The representatives of
 isomorphism classes of indecomposable objects of
a functorially finite maximal rigid subcategory form a basis in the
split Grothendieck group of another functorially finite maximal
rigid subcategory. As a direct consequence, the numbers of
indecomposable direct summands of functorially finite maximal rigid
objects are the same. In the last section, we prove that any
functorially finite maximal rigid subcategory is Gorenstein of
dimension at most $1$. Finally, for two reachable maximal rigid
objects, the corresponding endomorphism algebras have same
representation type.

\section{Relations between cluster-tilting subcategories and maximal rigid subcategories}

Throughout this paper, $k$ denotes an algebraically closed field and
$\C$ denotes a $k-$linear triangulated category whose shift functor
is denoted by $[1]$. We assume that $\C$ is $\Hom-$finite
 and Krull-Remak-Schmidt, i.e. $\dim _k\Hom(X,Y)<\infty $ for any two objects $X$ and $Y$ in $\C$, and every object can be written in a
unique way (up to isomorphism) as a finite direct sum of
indecomposable objects. For basic references on representation
theory of algebras and triangulated categories, we refer [H].
\medskip

 For $X,Y \in \C$ and
$n\in \mathbf{Z}$, we put

$$\Ext^n(X,Y)=\Hom(X, Y[n]).$$

 For a subcategory $\T$, we mean that $\T$ is a full subcategory which
 is closed under taking isomorphisms.  $\T^\bot$ denotes the subcategory consisting of $Y\in C$ with $\Hom(T,Y)=0$ for any $T\in \T$, and
 $^\bot T$ denotes the subcategory consisting of $Y\in C$ with $\Hom(Y,T)=0$ for any $T\in\T$. For an object $T\in C$, add$T$ denotes the subcategory consisting of direct summands of direct
 sums of finite copies of $T$.
\medskip

 For two subcategories $\X, \Y$, we denote $\Ext^1(\X, \Y)=0$ if $\Ext^1(X,Y)=0$  for any $X\in \X, Y\in
  \Y$. $\X*\Y$ denotes the extension category of $\X$ by $\Y$,
  whose objects are by definition the objects $E$ with triangle $X\rightarrow E\rightarrow Y\rightarrow X[1]$, where $X\in\X, Y\in
  \Y$. By the octahedral axiom, we have $(\X*\Y)*\Z=\X*(\Y*\Z)$. We
  call $\X$ extension closed if $\X*\X=\X$.
\medskip

  For $X\in \C$, a morphism
 $f:T\rightarrow X$ is called right $\T-$approximation of $X$ if
 $\Hom(-,f)|_{\T}$ is surjective. If any object $X\in \C$ has a right
 $\T-$approximation, we call $\T$  contravariantly finite in $\C$.
  Left $\T-$approximation and covariantly finiteness are defined
  dually. We say that $\T$ is functorially finite if it is both covariantly finite and contravariantly
  finite. It is easy to see that add$T$ is functorially finite for any object $T\in \C$.

\medskip

 A triangulated category $\C$ is called $2-$CY provided that there are bifunctorial isomorphisms

 $$\Ext^1(X,Y)=D\Ext^1(Y,X)$$ for $X, Y\in \C$, where $D=\Hom_k(-,k)$ is the
 duality of $k-$spaces.
\medskip

  An exact category is called stably $2-$CY [BIRS]
 if it is Frobenius, that is, it has enough projectives and injectives, which coincide, and the stable category is $2-$CY
  triangulated. If a triangulated category is triangulated
  equivalent to the stable category of a stably $2-$CY exact
  category, then we call it algebraic $2-$CY triangulated category [K3].
\medskip

 Examples of stably $2-$CY categories are the categories of Cohen-Macaulay modules over an isolated hypersurface singularity [BIKR];
 the module categories of preprojective algebras of Dynkin
 quivers [GLS]. Basic examples of $2-$CY triangulated categories are the cluster
 categories of abelian hereditary categories with tilting objects [BMRRT, Ke1]; the generalized cluster categories of algebras with
  global dimension of at most $2$
 [Am]; the stable categories of stably $2-$CY categories [BIRS] and
cluster tubes [BKL, BMV].
\medskip

We recall some basic notions [BMRRT, I1, KR, GLS1, BIRS].

\begin{defn}

Let $\T$ be a subcategory of $\C$ which is closed under taking
direct summands and finite direct sums.

\begin{enumerate}

\item $\T$ is called rigid provided $\Ext^1(\T,\T)=0$.

\item $\T$ is called maximal rigid provided $\T$ is rigid and is
maximal with respect to this property, i.e. if $\Ext^1(\T\bigcup\add
M,\T\bigcup \add M)=0$, then $M\in\T$.

\item $\T$ is called cluster-tilting provided $\T$ is functorially
finite and $\T={}^\bot\T[1]$.

\item An object $T$ is called rigid, maximal rigid, or cluster tilting  if $addT$ is rigid, maximal rigid, or cluster tilting respectively.
\end{enumerate}

\end{defn}

\begin{rem}
\begin{enumerate}

\item  Any $2-$CY triangulated category $\C$ admits rigid subcategories ($0$ is viewed as a trivial rigid object), and also admits
maximal rigid subcategories if $\C$ is skeletally small.

\item  There are $2-$CY triangulated categories which contains no
cluster tilting subcategories [BIKR, BMV].

 \item  Cluster tilting subcategories are functorially maximal
rigid subcategories. But the converse is not true in general. It was
observed by Buan-Marsh-Vatne [BIKR, BMV] that the cluster tubes
contain maximal rigid objects, none of them are cluster tilting
objects.

\end{enumerate}
\end{rem}

If $\C$ admits a cluster-tilting subcategory $\T$, we know that
$\C=\T*\T[1]$, i.e. for any object $X$ in $\C$ there is a triangle
$T_1\rightarrow T_0\rightarrow X\rightarrow T_1[1]$ with $T_i\in\T,
i=1,2$ [KR, KZ].
 In fact, the converse is true.

\begin{rem}

If $\C=\T*\T[1]$ with $\T$ a rigid subcategory of $\C$, then $\T$ is
a cluster-tilting subcategory.

\end{rem}

\begin{proof}

Clearly, $\T$ is functorially finite. Given an object $X$ in $\C$
with $\Ext^1(X,\T)=0$, there is a triangle $T_1\rightarrow
T_0\rightarrow X\rightarrow T_1[1]$. Then this triangle splits.
Hence $X\in\T$. This proves that $\T$ is cluster-tilting.

\end{proof}

In general, for a maximal rigid subcategory $\R$, $\R*\R[1]$ is
smaller than $\C$, but all rigid objects belong to $\R*\R[1]$
[BIRS]. The following lemma was proved for Preprojective algebras in
[BMR, GLS1], it holds for any $2-$CY triangulated category.

\begin{lem}

Let $\R$ be a contravariantly finite maximal rigid subcategory in a
$2-$CY triangulated category $\C$. For any rigid object $X\in \C$,
if $Y\s{f}\rightarrow R_0\s{g}\rightarrow X\s{h}\rightarrow Y[1]$ is
a triangle such that $R_0\in\R$ and $g$ is a right
$\R-$approximation of $X$, then $Y\in\R$. Furthermore, there is a
left $\R-$approximation $f_1: X\rightarrow R_2$, which
 extends a triangle $X\s{f_1}\rightarrow R_2\s{g_1}\rightarrow
R_3\s{h_1}\rightarrow X[1]$ with  $R_3\in\R$.

\end{lem}

\begin{proof}

Since $g$ is a right $\R-$approximation of $X$ and $\Ext^1(R,R_0)=0$
for any objects $R$ in $\R$, we have that
  $\Ext^1(\R,\add Y)=0$, in particular, we have $\Ext^1(R_0,Y)=0$.

By applying $\Hom(-,X)$ and $\Hom(Y,-)$ to the triangle
$Y\s{f}\rightarrow R_0\s{g}\rightarrow X\s{h}\rightarrow Y[1]$ we
have two exact sequences:

\begin{equation*}
\Hom(R_0,X)\s{\Hom(f,X)}\Longrightarrow\Hom(Y,X)\s{}\Longrightarrow
\Ext^1(X,X)=0
\end{equation*}
and
\begin{equation*}
\Hom(Y,R_0)\s{\Hom(Y,g)}\Longrightarrow\Hom(Y,X)\s{}\Longrightarrow
\Ext^1(Y,Y)\Longrightarrow\Ext^1(Y,R_0)=0.
\end{equation*}

Let $\alpha$ be an element of $\Hom(Y,X)$. By the first exact
sequence there is a $\beta\in\Hom(R_0,X)$ such that $\alpha=\beta
f$. Since $g$ is a right $\R-$approximation of $X$, there is a
$\gamma\in \Hom(R_0,R_0)$ such that $\beta=g\gamma$. Then
$\alpha=g\gamma f$. This shows that $\Hom(Y,g)$ is surjective. Hence
$\Ext^1(Y,Y)=0$ by the second exact sequence. It follows that $Y\in
\R$.

For the second part, we apply the first part to the rigid object
$X[1]$. There is a triangle $R_3\s{f_1}\rightarrow
R_4\s{g_1}\rightarrow X[1]\s{h_1}\rightarrow R_4[1]$. Then we have a
triangle $X\s{-h_1[-1]}\rightarrow R_4\rightarrow R_3\rightarrow
X[1]$. It is easy to see $-h_1[-1]$ is a left $\R-$approximation of
$X$.

\end{proof}

There is a dual statement for covariantly finite maximal rigid
subcategory $\R$, we leave it to the reader.

By Lemma 2.1, we have result which is the second part of Proposition
I.1.7 in [BIRS].

\begin{cor}

Let $\R$ be a functorially finite maximal rigid subcategory. Then
every rigid object belongs to $\R*\R[1]$.

\end{cor}

One can see that any cluster-tilting subcategory is maximal rigid,
but the converse is not true [BIKR, BMV]. The main result of this
section is the following theorem which tells us that the
cluster-tilting and maximal rigid can not really coexist. This is a
generalization of Theorem II.1.8 in [BIRS], where the same
conclusion was proved for algebraic $2-$CY triangulated categories.

\begin{thm}

If $\C$ admits a cluster-tilting subcategory $\T$, then every
functorially finite maximal rigid subcategory is cluster-tilting.

\end{thm}

\begin{proof}

Assume that $\R$ is a functorially finite maximal rigid subcategory
in $\C$. Given an object $X\in\C$ satisfying $\Ext^1(\add X,\R)=0$,
we have a triangle

\begin{equation*}
T_1\s{f}\Longrightarrow T_0\s{g}\Longrightarrow X\Longrightarrow
T_1[1]
\end{equation*}

where $T_0$ and $T_1$ belong to $\T$. Since $\R$ is functorially
finite in $\C$, there is a left $\R-$approximation of $T_0$ which
extends to a triangle by Lemma 2.4,

\begin{equation*}
R_0[-1]\Longrightarrow T_0\s{\alpha}\Longrightarrow R\Longrightarrow
R_0,
\end{equation*}
where $R,R_0\in \R.$

 Let $\alpha_1=\alpha f$. For any object $Z$ in
$\R$, by applying $\Hom(-,Z)$ to the triangle $T_1\s{f}\rightarrow
T_0\s{g}\rightarrow X\rightarrow T_1[1]$, we have the exact sequence

\begin{equation*}
\Hom(T_0,Z)\s{\Hom(f,Z)}\Longrightarrow
\Hom(T_1,Z)\s{}\Longrightarrow \Ext^1(X,Z)=0.
\end{equation*}

Given an element $\varphi_1\in\Hom(T_1,Z)$, there is a
$\varphi_0\in\Hom(T_0,Z)$ such that $\varphi_1=\varphi_0 f$. Since
$\alpha$ is a left $\R-$approximation of $T_0$, there is a $\psi$
such that $\varphi_0=\psi\alpha$. Then $\varphi_1=\psi\alpha
f=\psi\alpha_1$. So $\alpha_1$ is a left $\R-$approximation of
$T_1$. It follows from Lemma 2.4 that the triangle which $\alpha_1$
is a part is of the form:

\begin{equation*}
R_1[-1]\Longrightarrow T_1\s{\alpha_1}\Longrightarrow
R\Longrightarrow R_1,
\end{equation*}
where $R,R_1\in \R.$

Starting with $\alpha_1=\alpha f$, we get the following commutative
diagram by the octahedral axiom:

$$\begin{array}{ccccccc}
X[-1]&=&X[-1]\\
\downarrow&&\downarrow\\
R_1[-1]&\s{}\longrightarrow&T_1&\s{\alpha_1}\longrightarrow&R&\s{}\longrightarrow&R_1\\
\downarrow&&\downarrow f&&\parallel&&\downarrow\\
R_0[-1]&\s{}\longrightarrow&T_0&\s{\alpha}\longrightarrow&R&\s{}\longrightarrow&R_0\\
\downarrow&&\downarrow g\\
X&=&X

\end{array}$$

But $\Hom(R_0[-1],X)=\Ext^1(R_0,X)=0$, so the first column is a
split triangle and then $X\in\R$. Thus we have proved this theorem.

\end{proof}

\begin{rem}
The same conclusion is not true in arbitrary triangulated
categories. See the example in Section 2 in [BMRRT], where the
derived category of the quiver $Q:1\rightarrow 2\rightarrow 3$
contains a functorially finite maximal rigid subcategory which is
not cluster tilting. It is well-known that the derived category of
$Q$ contains cluster tilting subcategories, see for example the
example in Section 5 of [KZ], or [I2].
\end{rem}

\section{Mutations and basis of Grothendieck groups of maximal rigid subcategories}

Mutations in arbitrary triangulated categories were defined in [IY].
We recall them in the setting of $2-$CY triangulated categories.
\medskip

Let $\C$ be a $2-$CY triangulated category and $\D$ a functorially
finite rigid
 subcategory of $\C$ which is closed under taking finite direct sums and direct summands. For any subcategory
$\X\supset \D$, put

\begin{center}
$\mu ^{-1}(\X, \D):=(\D\ast \X [1])\cap {}^{\bot}(\D [1])$.
\end{center}

Dually, for a subcategory $\Y\supset \D,$ put

\begin{center}
  $\mu (\Y, \D):=(\Y[-1]\ast \D)\cap (\D [-1])^{\bot}$.\end{center}

\begin{defn} The pair $(\X,\Y)$ of subcategories $\X, \Y$ is called
$ \D-$mutation pair if $\X = \mu (\Y, \D)$ and $\Y=\mu ^{-1}(\X,
\D)$ [IY].

\end{defn}

It is not difficult to see that: for subcategories $\X, \Y$
containing $\D$, $(\X,\Y)$ forms a $\D-$mutation if and only if for
any $X\in \X$, $Y_1\in \Y$ there are two triangles:

\begin{equation*}
X\s{f}{\rightarrow}D\s{g}{\rightarrow}Y\rightarrow X[1],
\end{equation*}
\begin{equation*}
X_1\s{f_1}{\rightarrow}D_1\s{g_1}{\rightarrow}Y_1\rightarrow
X_1[1]\end{equation*}

where $D,D_1\in \D$, $Y\in \Y$, $X_1\in \X$,  $f$ and $f_1$ are left
$\D-$approximations; $g$ and $g_1 $ are right $\D-$approximations.
\medskip

The following result is analogous to the first part of Theorem 5.1
in [IY], where the arguments are stated for cluster tilting
subcategories. We give a proof here for the convenience of the
reader.

\begin{prop} Let $\R$ be a functorially finite maximal rigid subcategory containing $\D$
as a subcategory. Then its mutation $\R'=\mu ^{-1}(\X, \D)$ is a
functorially finite maximal rigid subcategory, and  $(\R, \R')$ is a
$\D-$mutation pair. \end{prop}

\begin{proof} Let $\Z={}^\bot\D[1]=\D[-1]^\bot$, and  $\U:=\Z/\D$ the quotient triangulated category,
whose shift functor is denoted by $<1>$ [IY]. The images of the
object $X$ and the
 subcategory $\X$ in the quotient $\U$ are denoted by $\underline{X}$ and $\underline{\X}$ respectively.

  We first sketch the proof of the fact $\X$ is a functorially finite
maximal rigid subcategory in $\C$ if and only if so $\underline{\X}$
is in $\U$.

 It is easy to see that $\X$ is functorially finite if and only if  $\underline{\X}$ is so.

  We will prove that $\underline{\X}$ is maximal rigid
in $\U$ provided $\X$ is maximal rigid in $\C$.

 Let $M\in \U$ satisfy that $\Hom(M, X<1>)=0, \Hom(X,M<1>)=0,
\Hom(M,M<1>)=0$, for any $X\in \underline{\X}$. For $X\in \X$, we
have a triangle
\begin{equation*}
X\s{f}\Longrightarrow D\s{g}\Longrightarrow X<1>\s{h}\Longrightarrow
X[1].
\end{equation*}

Now suppose that $\alpha\in \Hom(M,X[1])$. Since $\Hom(M,D[1])=0$,
$\alpha$ factors through $h$ by $\beta: M\rightarrow X<1>$. Since
$\Hom(M,X<1>)=0$ in $\U$, we have that $\beta$ factors through $g$.
Then $\alpha=0$. This proves that $\Hom(M,X[1])=0$. One can prove
that $\Hom(X[1],M)=0, \Hom(M,M[1])=0$ in a similar way. Then
$\underline{\X}$ is maximal rigid in $\U$. The converse implication
can be proved in a similar way. We omit the details here.

It follows from the fact $(\R,\R')$ is a $\D-$mutation in $\C$ that
$(\underline{\R}, \underline{\R'})$ is a $0-$mutation in $\U$. Then
$\underline{\R'}= \underline{\R}<1>$ in $\U$, and then
$\underline{R'}$ is maximal rigid in $\U$. It follows that $\R'$ is
 maximal rigid in $\C$.

\end{proof}
\medskip

We call a subcategory $\R_1$ an almost complete maximal rigid
subcategory if there is an indecomposable object $R$ which is not
isomorphic to any object in $\R_1$ such that $\R=\add(\R_1\bigcup
\{R\})$ is a functorially finite maximal rigid subcategory in $\C$.
 Such $R$ is called a complement of an almost complete maximal rigid subcategory
$\R_1$. It is easy to see that any almost complete maximal rigid
subcategory is functorially finite. Combining the proposition above
with Corollary 2.5, we have the following corollary, which was
indicated in [BIRS].

\begin{cor}
Let $\R_1$ be an almost complete maximal rigid subcategory of $\C$.
Then there are exactly two complements of $\R_1$, say $R$ and $R^*$.
Denote by $\R=\add({\R_1}\bigcup \{R\})$, $\R'=\add({\R_1}\bigcup
\{R^*\})$. Then $(\R,\R'), (\R',\R)$ are $\R_1-$mutations.
\end{cor}

\begin{proof} This follows from [IY, 5.3]. Note that the arguments
there are stated only for cluster tilting subcategories, but work
also for functorially finite maximal rigid subcategories with the
help of Corollary 2.5.

\end{proof}

\begin{defn} Let $\R_1$ be an almost complete maximal rigid subcategory of $\C$
with the complements $R$ and $R^*$. Denote $\R=\add(\R_1
\bigcup\{R\})$, $\R'=\add(\R_1 \bigcup\{R^*\})$. If
$\dim_k\Hom(R,R^*[1])=1$, then the mutation $(\R,\R')$ is called a
 simple mutation [Pla].
\end{defn}

There are mutations of some maximal rigid objects in cluster tubes
which are not simple [Y]. It was proved in [Pla] that for a simple
mutation of maximal rigid subcategories $\R, \R'$,
$\R*\R[1]=\R'^*\R'[1]$.

\medskip

Let $\R$ be a functorially maximal rigid subcategory of $2-$CY
triangulated category $\C$. Let $K_0^{split}(\R)$ be the (split)
Grothendieck group, which by definition, the free abelian group with
a basis $[R]$, where $R$ runs from the representatives of
isomorphism classes of indecomposable objects in $\R$.  Let $X$ be a
rigid object of $\C$. By the corollary 2.5 above, there is a
triangle $\R_1\rightarrow\R_0\rightarrow X\rightarrow R_1[1]$. So we
can define the index $\ind_\R(X)=[R_0]-[R_1]\in K_0^{split}(\R)$ as
in [Pa, DK, Pla].

\begin{prop}

Let $\R$ be a functorially finite maximal rigid subcategory.

(1) If $X$ and $Y$ are rigid objects in $\C$ such that
$ind_\R(X)=ind_\R(Y)$, then $X$ and $Y$ are isomorphic.

(2) Let $X$ be a rigid object of $\C$ and let $X_i, i\in I$, be a
finite family of pairwise nonisomorphic indecomposable direct
summands of $X$. Then the elements $\ind_\R(X), i\in I$, are
linearly independent in $K_0^{split}(\R)$.

\end{prop}

\begin{proof}

All conclusions follow from Sections 2.1, 2.2, 2.3, 2.5 in [DK].
Note that the arguments there are stated only for cluster-tilting
subcategories, but work also for functorially finite maximal rigid
subcategories by the Corollary 2.5 above.

\end{proof}

\begin{thm}

Let $\R'$ be another functorially finite maximal rigid subcategory
in $\C$. Then the elements $\ind_{\R}(R')$, where $R'$ runs through
a system of representatives of the isomorphism classes of
indecomposables of $\R'$, form a basis of the free abelian group
$K_0^{split}(\R)$.

\end{thm}

\begin{proof}

The proof of Theorem 2.4 in [DK, 2.6] works also in this setting.

\end{proof}

\begin{cor} \begin{enumerate}\item The category $\C$ has a maximal
rigid object if and only if all functorially maximal rigid
subcategories have a finite number of pairwise non-isomorphic
indecomposable objects. \item  All maximal rigid objects have the
same number of indecomposable direct summands (up to
isomorphism).\end{enumerate}
\end{cor}

\section{Gorenstein property of maximal rigid subcategories}

Let $\R$ be a functorially finite maximal rigid subcategory of $\C$
and $\A$ the quotient category of $\D=\R[-1]*\R$ by $\R$. Let
mod$\R$ denote the category of finitely presented $\R-$modules where
a $\R-$module means a contravariantly functor from $\R$ to the
category of $k-$vector spaces. We know that $\A$ is an abelian
category whose abelian structure is induced by the triangulated
structure of $\C$ and there is an equivalence
$F:\A\rightarrow\text{mod}\R$ [IY]. As in section 2, we put
\begin{equation*}
\R^{\bot}:=\{X\in\C\mid \Hom(\R,X)=0\}\text{ and
}{}^{\bot}\R:=\{X\in\C\mid \Hom(X,\R)=0\}.
\end{equation*}
By $2-$CY property of $\C$,
 ${}^\bot\R[1]=\R[-1]^{\bot}$, which is denoted by $\S$. Clearly, both
($\R,\S$) and ($\S,\R$) are cotorsion pairs in the sense in [N]
(equivalently, $(\R, \S[1])$ and $(\S[-1],\R)$ are torsion pairs in
the sense in [IY]).
\medskip

For the convenience of the reader we recall briefly the abelian
structure of $\A$ from [N]. Let $\underline{f}\in\Hom_\A(X,Y)$ with
$X,Y\in\D$ and $f\in\Hom_\C(X,Y)$ where $f$ is a part of the
triangle $Z[-1]\s{h}\rightarrow X\s{f}\rightarrow Y\s{g}\rightarrow
Z$. Let $f_1: X\rightarrow R_0$ be a left $\R-$approximation of $X$
which extends to a triangle $R_1[-1]\rightarrow X\s{f}\rightarrow
R_0\rightarrow R_1$. Then $R_1\in \R$ by Lemma 2.4. We have the
following commutative diagram which is constructed from the octahedral
axiom:

$$\begin{array}{ccccccc}
&&Z[-1]&=&Z[-1]\\
&&\downarrow h&&\downarrow\\
R_1[-1]&\s{}\longrightarrow&X&\s{}\longrightarrow&R_0&\s{}\longrightarrow&R_1\\
\parallel&&\downarrow f&&\downarrow&&\parallel\\
R_1[-1]&\s{}\longrightarrow&Y&\s{m_f}\longrightarrow&M_f&\s{}\longrightarrow&R_1\\
&&\downarrow g&&\downarrow \\
&&Z&=&Z
\end{array}(*).$$
 The map $\underline{m_f}$ is the cokernel of
$\underline{f}$ [N] (note that $M_f\in
(\R[-1]*\R)*\R=\R[-1]*(\R*\R)=\R[-1]*\R)$.
\medskip

The kernel of $\underline{f}$ is obtained similarly. Since
$(\R,\S[1])$ is a torsion pair, we have a triangle $S_0'\rightarrow
R_0'\rightarrow Y\rightarrow S_0'[1]$, where $R_0'\in\R$ and
$S_0'\in\S$. Using the octahedral axiom, we have the first diagram
of the following two commutative diagrams. Since $(\R[-1], \S)$ is a
torsion pair, we have a triangle $R[-1]\s{\psi}\rightarrow
L_f\rightarrow S\rightarrow R$ with $R\in\R$ and $S\in\S$. Since
$(\R, \S[1])$ is also a torsion pair, we have another triangle
$S'\s\rightarrow R'\rightarrow S\rightarrow S'[1]$ with $R'\in\R$
and $S'\in\S$. Using the octahedral axiom, we have the second
diagram of the following two commutative diagrams:

$$\begin{array}{ccccccc}
&&Z[-1]&=&Z[-1]\\
&&\downarrow&&\downarrow h\\
S_0'&\s{}\longrightarrow&L_f&\s{l_f}\longrightarrow&X&\s{}\longrightarrow&S_0'[1]\\
\parallel&&\downarrow&&\downarrow f&&\parallel\\
S_0'&\s{}\longrightarrow&R_0'&\s{}\longrightarrow&Y&\s{}\longrightarrow&S_0'[1]\\
&&\downarrow&&\downarrow g\\
&&Z&=&Z
\end{array}(**)$$

and

$$\begin{array}{ccccccc}
&&S'&=&S'\\
&&\downarrow&&\downarrow\\
R[-1]&\s{\varphi}\longrightarrow&K_L&\s{}\longrightarrow&R'&\s{}\longrightarrow&R\\
\parallel&&\downarrow k_L&&\downarrow&&\parallel\\
R[-1]&\s{\psi}\longrightarrow&L_f&\s{}\longrightarrow&S&\s{}\longrightarrow&R\\
&&\downarrow&&\downarrow \\
&&S'[1]&=&S'[1]
\end{array}(***).$$

 The composition $\underline{l_fk_L}$ is the kernel of
$\underline{f}$.

\begin{rem}

[N, Remark 4.5]For any $X\in\R[-1]*\R$ and any
$\underline{y}\in\Hom_{\C/\R}(X,L_f)$, there exists a unique
morphism $\underline{x}\in\Hom_\A(X,K_L)$ such that
$\underline{y}=\underline{k_L} \underline{x}$.

$$\xymatrix@R=0.5cm{
                &         K_L \ar[dd]^{\underline{k_L}}     \\
  X \ar[ur]^{\underline{x}} \ar[dr]_{\underline{y}}                 \\
                &         L_f                 }$$
Thus $K_L$ is  determined uniquely up to a canonical isomorphism in
$\A$.
\end{rem}

The following lemma is a suitable version of Proposition 6.1 in [N]
in our setting. We include a proof for the convenience of the
reader.

\begin{lem}
Let $f:X\rightarrow Y$ be a morphism in $\C$ which extends a
triangle $Z[-1]\s{h}\rightarrow X\s{f}\rightarrow Y\s{g}\rightarrow
Z$. Then $\underline{f}$ is an epimorphism if and only if
$M_f\in\R$; $\underline{f}$ is a monomorphism if and only if
$L_F\in\S$.

In particular, if $Z$ is in $\R$, then $f$ is an epimorphism; if
$Z[-1]$ is in $\S$, then $f$ is a monomorphism.
\end{lem}

\begin{proof}

(1) $\underline{f}:X\rightarrow Y$ is an epimorphism in $\A$ if and
only if $M_f\cong0$ in $\A$, i.e. $M_f\in\R$.

(2) $\underline{f}:X\rightarrow Y$ is a monomorphism in $\A$ if and
only if $K_L\cong0$ in $\A$, i.e. $K_L\in\R$. We claim that
$K_L\in\R$ if and only if $L_F\in\S$. If $K_L\in\R$, then
$\varphi=0$ by $\Hom(\R[-1],\R)=0$. So $\psi=k_L\varphi=0$, and
hence $L_F\in\S$. If $L_F\in\S$, then $\psi=0$. So $k_L\varphi=0$
that implies that $k_L$ factors through $R'$. Then
$\underline{k_L}=\underline{0}$, and then $K_L\in\R$ by Remark 4.1.

(3)If $Z$ is in $\R$, then the third column in (*) is a splitting
triangle. So $M_f\in\R$. Dually, if $Z[-1]$ is in $\S$, then the
second column in (**) is a splitting triangle. So $L_f\in\S$.

\end{proof}

Now we determine the projective objects and injective objects in
$\A$.

\begin{prop}

An object $M$ of $\A$ is a projective object if and only if
$M\in\R[-1]$. An object $N$ of $\A$ is an injective object if and
only if $N\in\R[1]$.

\end{prop}

\begin{proof}

(1)Given $R\in\R$. For any epimorphism $\underline{f}:X\rightarrow
Y$ in $\A$, and any morphism $\underline{\alpha}:R[-1]\rightarrow
Y$, $m_f\alpha=0$ by $M_f\in\R$ and $\Hom(\R[-1],\R)=0$. So
$g\alpha=0$. Then $\alpha$ factors through $f$, hence
$\underline{\alpha}$ factors through $\underline{f}$. This proves
that $R[-1]$ is projective in $\A$.

Conversely assume $M$ is a projective object in $\A$. Since
$M\in\R[-1]*\R$, there is a triangle $\R_0[-1]\s{\sigma}\rightarrow
M\rightarrow R_1\rightarrow R_0$ with $R_0,R_1\in\R$. Then
$\underline{\sigma}$ is an epimorphism in $\A$ by Lemma 4.2. So the
epimorphism $\underline{\sigma}: \R_0[-1]\rightarrow M   $ splits.
Hence $M\in\R[-1]$.

(2)Note that $R[1]\in\R[-1]*\R$, for all $R\in\R$, by Corollary 2.5.

Given $R\in\R$. For any monomorphism $\underline{f}:X\rightarrow Y$
in $\A$, and any morphism $\underline{\beta}:X\rightarrow R[1]$,
$\beta l_f=0$ by $L_f\in\S$. So $\beta h=0$. Then $\beta$ factors
through $f$, hence $\underline{\beta}$ factors through
$\underline{f}$. This proves that $R[1]$ is injective in $\A$.

Conversely assume $M$ is an injective object in $\A$. Since
$M\in\C=\S*\R[1]$, there is a triangle $\S\rightarrow
M\s{\tau}\rightarrow R[1]\rightarrow S[1]$ with $R\in\R$ and
$S\in\S$. Then $\underline{\tau}$ is a monomorphism in $\A$ by Lemma
4.2. So $\underline{\tau}$  splits, hence $M\in\R[1]$.

\end{proof}

The main result in this section is the following theorem which is a
generalization of Proposition 2.1 in [KR], Theorem 4.3 in [KZ]. This
has been proved in [V, Y] for $\C$ being cluster tubes.

\begin{thm}

Let $\C$ be a 2-Calabi-Yau triangulated category with a functorially
finite maximal rigid subcategory $\R$ and let $\A$ be the abelian
quotient category of $\R[-1]*\R$ by $\R$. Then

(1) The abelian category $\A$ has enough projective objects.

(2) The abelian category $\A$ has enough injective objects.

(3) The abelian category $\A$ is Gorenstein of Gorenstein dimension
at most one.

\end{thm}

\begin{proof}

(1) Given $X\in\R[-1]*R$. There is a triangle $R_1[-1]\rightarrow
R_0[-1]\s{f}\rightarrow X\rightarrow R_1$ with $R_0,R_1\in\R$. Then
$\underline{f}:R_0[-1]\rightarrow X$ is an epimorphism with
$R_0[-1]$ a projective object.

(2) Given $X\in\R[-1]*R$. Since $X\in\C=\S*\R[1]$. There is a
triangle $S\rightarrow X\s{g}\rightarrow R[1]\rightarrow S[1]$ with
$R\in\R$ and $S\in\S$. Then $\underline{g}:X\rightarrow R[1]$ is a
monomorphism with $R[1]$ an injective object.

(3) For an injective object $R[1]$ in $\A$, since
$R[1]\in\R[-1]*\R$, then there is a triangle
$R_1[-1]\s{h}\rightarrow R_0[-1]\s{f}\rightarrow R[1]\rightarrow
R_1$ with $R_0,R_1\in\R$. Then $\underline{f}$ is an epimorphism by Lemma 4.2 and
$\underline{h}$ is the kernel of $\underline{f}$ by the structure of
kernel. So we have an exact sequence $0\rightarrow
R_1[-1]\s{\underline{h}}\rightarrow
R_0[-1]\s{\underline{f}}\rightarrow R[1]\rightarrow0$ which is a
projective resolution of the injective object $R[1]$ in $\A$.
Therefore proj.dim.$R[1]\leq1$.

For a projective object $R[-1]$ in $\A$, since $R[-1]\in\R*\R[1]$ by
Corollary 2.5, there is a triangle $R_0\rightarrow R[-1]\s{f}\rightarrow
R_1[1]\s{g}\rightarrow R_0[1]$ with $R_0,R_1\in\R$. Then $\underline{f}$ is a
monomorphism in by Lemma 4.1 and $\underline{g}$ is the cokernel of
$\underline{f}$ by the structure of cokernel. So we have an exact
sequence $0\rightarrow R[-1]\s{\underline{f}}\rightarrow
R_1[1]\s{\underline{g}}\rightarrow R_0[1]\rightarrow0$ which is a
injective resolution of the projective object $R[-1]$ in $\A$.
Therefore inj.dim.$R[-1]\leq1$.

Therefore $\A$ is Gorenstein of Gorenstein dimension at most one.

\end{proof}

As in [KZ], we have the following corollary.

\begin{cor}

Let $\C$ be a 2-Calabi-Yau triangulated category and $\R$ a
functorially finite maximal rigid subcategory. Then $\A$ is a
Frobenius category if and only if $\R=\R[2]$.

\end{cor}

\begin{proof}
$\A$ is Frobenius if and only if the sets of projective objects or
of injective objects of $\A$ coincide, i.e. $\R[-1]=\R[1]$ if and
only if $\R=\R[2]$.

\end{proof}

In the last part of this section, we assume that the $2-$CY
triangulated category $\C$ admits a maximal rigid object. It follows
from Corollary 3.7 that all maximal rigid subcategories are of form
add$R$, where $R$ is a maximal rigid object. The numbers of
indecomposable direct summands of all basic maximal rigid objects
are the same. Two maximal rigid objects are called reachable via
simple mutations if one of them can be obtained from another by
finite steps of simple mutations.

\begin{defn}
Let $R$ be a maximal rigid object in $\C$. The endomorphism algebra
of $R$ is called a $2-$CY maximal tilted algebra.

\end{defn}

The $2-$CY tilted algebras [BIRS] which by definition the
endomorphism algebras of cluster tilting objects in a $2-$CY
triangulated category are a special case of $2-$CY maximal tilted
algebra. The converse is not true in general since the $2-$CY
maximal tilted algebras may contains loops [BIKR][BRV].

Now we collect the representation theoretic properties of $2-$CY
maximal tilted algebras.

\begin{prop}
\begin{enumerate}
\item $2-$CY maximal tilted algebras are Gorenstein algebras of
dimension at most $1$.
\item Let $R$ and $R'$ form a simple mutation pair. Then $\End R$ and
$\End R'$ are nearly Morita equivalent, i.e. $\mod R/\add S_i\approx
\mod R'/add S'_i$.
\item If $R$ and $R'$ are reachable via simple mutations, then $\End R$ and
$\End R'$ have the same representation type.

\end{enumerate}

\end{prop}

\begin{proof}
\begin{enumerate}
\item This is direct consequence of Theorem 4.4.

\item This was proved in [Y].

\item Denote $A=\End R$ and $A'=\End R'$. From the assumption, we have that $R[-1]*R=R'[-1]*R'$ by
[Pal], which is denoted by $\D$.  By Theorem 4.4,
$A$-mod$\thickapprox \D/ \mbox{add}R$ and $A'$-mod$\thickapprox \D/
\mbox{add}R'$. Therefore $A$-$\mbox{mod}/\mbox{add}R'\thickapprox
\mathcal{D}/ \mbox{add}(R\cup R')\approx
A'$-$\mbox{mod}/\mbox{add}R.$ Hence, ind$A$ is a finite set if and
only if ind$\mathcal{D}$ is a finite set. Thus $A$ is of finite type
if and only if $A'$ is so. Moreover, by the proof in [Kr], $A
$-$\mbox{mod}$ is wild if and only if $A'$-$\mbox{mod}$ is wild.
Therefore, by tame-wild dichotomy, $A$ and $A' $ have the same
representation type.

\end{enumerate}

\end{proof}

\begin{center}
\textbf {ACKNOWLEDGMENTS.}\end{center} The authors would like to
thank Idun Reiten for her suggestion and interesting in this work.


\begin{thebibliography}{99}

\bibitem[Am]{Am}
C.Amiot.
\newblock Cluster categories for algebras of global dimension $2$ and quivers with potential.
\newblock Ann. Inst. Fourier \textbf{59}, 2525-2590, 2009.



\bibitem[BKL]{BKL}
M.Barot. Dirk Kussin and H.Lenzing.
\newblock The Grothendieck group of a cluster category.
\newblock Jour. Pure and App. Algebra. \textbf{212}, 33-46, 2008.




\bibitem[BIKR]{BIKR}
I.Burban,O.Iyama. B.Keller and I.Reiten.
\newblock Cluster tilting for one-dimensional hypersurface singularities.
\newblock Adv. Math. \textbf{217}, 2443-2484, 2008.


\bibitem[BM]{BM}
A.Buan and R.Marsh.
\newblock Cluster-tilting theory.
\newblock Trends in representation theory of algebras and related topics.
Edited by J.de la Pe$\tilde{n}$a and R. Bautista. Contemporary Mathematics. \textbf{406}, 1-30, 2006.



\bibitem[BMR]{BMR}
A.Buan, R.Marsh and I.Reiten.
\newblock Cluster-tilted algebras.
\newblock Trans. AMS. \textbf{359}, 323-332, 2007.


\bibitem[BIRS]{BMRRT}
A.Buan, O.Iyama, I.Reiten and J.Scott.
\newblock Cluster structure for $2-$Calabi-Yau categories and unipotent groups.
\newblock Compo. Math. \textbf{145}(4), 1035-1079, 2009.


\bibitem[BMV]{BMV}
A.Buan, R.Marsh and D.Vatne.
\newblock Cluster structure from $2-$Calabi-Yau categories with loops.
\newblock To appear in Math.Zeite., see also arxiv:math.RT/0810.3132, 2008.


\bibitem[BMRRT]{BMRRT}
A.Buan, R.Marsh, M.Reineke, I.Reiten and G.Todorov.
\newblock Tilting theory and cluster combinatorics.
\newblock Advances in Math. \textbf{204}, 572-618, 2006.


\bibitem[CC]{CC}
P.Caldero and F.Chapoton.
\newblock Cluster algebras as Hall algebras of quiver representations.
\newblock Comment. Math. Helv. \textbf{81}, 595-616, 2006.


 \bibitem[CK]{CK}
P.Caldero and B.Keller.
\newblock From triangulated categories to cluster algebras II.
\newblock Ann.Sci.Ecole.Norm.Sup.(4) \textbf{39}, 983-1009,2006.


\bibitem[DK]{DK} Raika Dehy and Bernhard Keller
\newblock On the Combinatorics of Rigid Objects in 2-Calabi-Yau Categories.
\newblock International Mathematics Research Notices, Vol.2008, Article ID run029, 17 pages.


\bibitem[FZ1]{FZ1}
S.Fomin and A.Zelevinsky.
\newblock Cluster Algebras I: Foundations.
\newblock Jour. Amer. Math. Soc. \textbf{15}, no. 2, 497--529, 2002.


\bibitem[FuKe]{FuKe}
Chagjian Fu and Bernhard Keller.
\newblock On cluster algebras wth coefficients and $2-$Calabi-Yau categories.
\newblock Trans.Amer.Math.Soc.\textbf{362},859-895, 2010.


\bibitem[GLS1]{GLS1}
Ch.Gei$\beta$, B.Leclerc and J.Shr$\ddot{o}$er.
\newblock Rigid modules over preprojective algebras.
\newblock Invent. math. \textbf{165}(3), 589-632, 2006.


\bibitem[GLS2]{GLS2}
Ch.Gei$\beta$,B.Leclerc and J.Shr$\ddot{o}$er.
\newblock Preprojective
algebras and cluster algebras.
\newblock Trends in representation theory of
algebras and related topics, 253-283, EMS Ser.Congr.Rep.
Eur.math.Soc. Z$\ddot{u}$rich,2008.


\bibitem[H]{H}
D.Happel.
\newblock Triangulated categories in the representation theory of quivers.
\newblock LMS Lecture Note Series, 119. Cambridge, 1988.


\bibitem[I1]{I1}
O.Iyama.
\newblock Higher dimensional Auslander-Reiten theory on maximal orthogonal subcategories.
\newblock Adv. Math. 210(2007), 22-50.


\bibitem[I2]{I2}
O.Iyama.
\newblock Cluster tilting for higher Auslander algebras.
\newblock To appear in Adv. Math. See also arXiv:0809.4897v3.


\bibitem[IY]{IY}
O.Iyama and Y. Yoshino
\newblock Mutations in triangulated categories and rigid Cohen-Macaulay modules.
\newblock Invent. Math. \textbf{172}, 117¨C168, 2008.


\bibitem[Ke1]{Ke1}
B.Keller.
\newblock Triangulated orbit categories.
\newblock Documenta Math. \textbf{10}, 551-581, 2005.


\bibitem[Ke2]{Ke2}
B.Keller.
\newblock Cluster algebras, quiver representations and
triangulated categories.
\newblock arXiv: 0807.1960.




\bibitem[Ke3]{Ke3}
B.Keller.
\newblock Triangulated Calabi-Yau categories, Trends in Representation Theory of Algebras (Zurich)
(A. Skowro¡änski, ed.), European Mathematical Society, 2008, pp.
467¨C489.

\bibitem[KR1]{KR}
B.Keller and I.Reiten.
\newblock Cluster-tilted algebras are Gorenstein and stably Calabi-Yau.
\newblock Adv.
Math. 211, 123-151, 2007.


\bibitem[KZ]{KZ}
S.Koenig, and B.Zhu.
\newblock From triangulated categories to abelian categories-- cluster tilting in a  general framework.
\newblock Math. Zeit. 258, 143-160, 2008.


\bibitem[Kr]{Kr}
H. Krause.
\newblock Stable equivalence preserves representation type.
\newblock Comment. Math.Helv. \textbf{72}, 266-284, 1997.


\bibitem[Na]{Na}
H. Nakaoka.
\newblock  General heart construction on a triangulated category (I):unifying $t-$structures and cluster tilting subcategories.
To appear in Appl.Category structure, 2010. See also
arXiv:0907.2080(V6).


\bibitem[Pa1]{Pa1}
Y. Palu.
\newblock  Cluster characters for $2-$Calabi-Yau triangulated categories.
\newblock Ann.Inst.Fourier, \textbf{58}, 133-171, 2008.


\bibitem[Pa2]{Pa2}
Y. Palu.
\newblock  Grothendieck group and generalized mutation rule for $2-$Calabi-Yau triangulated categories.
\newblock Jour.of Pure and Appl. Alg.
 \textbf{213}, 1438-1449, 2009.


\bibitem[Pla]{Pla}
P. Plamondon.
\newblock  Cluster characters for cluster categories with infinite-dimensional morphism spaces.
\newblock arXiv. 1002.4956.


\bibitem[Rin]{Rin}
C.~M.~Ringel.
\newblock Some remarks concerning tilting modules and tilted
algebras. Origin. Relevance. Future. An appendix to the Handbook of
tilting theory, edited by Lidia Angeleri-H\"ugel, Dieter Happel and
Henning Krause. Cambridge University Press (2007), LMS Lecture Notes
Series 332.


\bibitem[V]{V}
D.Vatne.
\newblock Endomorphism rings of maximal rigid objects in cluster tubes.
\newblock arXiv: 0905.1796.


\bibitem[Y]{Y}
D.Yang.
\newblock Endomorphism algebras of maximal rigid objects in cluster tubes.
\newblock arXiv: 1004.1303.


\end{thebibliography}
\end{document}